\documentclass{llncs}

\usepackage{llncsdoc}

\usepackage{amsmath}
\usepackage{amsfonts}
\usepackage{amssymb}
\usepackage{latexsym}
\usepackage{exscale,cmmib57}
\usepackage{graphicx,color}
\usepackage{multirow}
\usepackage{graphicx}
\usepackage{epic}
\usepackage{subfigure}
\usepackage{array}
\usepackage{epic,eepic}
\usepackage{fancybox}
\usepackage{psfrag}
\usepackage{caption}
\usepackage{float}
\usepackage{multirow}

\def \b{{\bf b}}

\def \bz{{\bf z}}

\newcommand{\mean}[1]{\{ \kern -1.6mm \{#1\} \kern -1.6mm \}}
\newcommand{\jump}[1]{[ \kern -.7mm [#1] \kern -.7mm ]}

\newcommand{\be}{\begin{equation}}\newcommand{\ee}{\end{equation}}

\newcommand{\ndg}[1]{| \kern -.25mm \|{#1}| \kern -.25mm \|}

\renewcommand{\ldots}{\dotsc}

\def\card{{\rm card}\,}

\def \bl{{\bf l}}
\def\b0{{\bf 0}}
\def\bi{{\bf i}}
\def\bx{{\bf x}}
\def\bj{{\bf j}}
\def\bm{{\bf m}}
\def\bn{{\bf n}}
\def\bk{{\bf k}}
\def\ba{{\bf a}}
\def\by{{\bf y}}
\def\erf{{\rm erf}\,}

\begin{document}

\title{Fast multilevel sparse Gaussian kernels \\ for high-dimensional approximation \\ and integration}

\author{
Zhaonan Dong, Emmanuil H. Georgoulis, Jeremy Levesley, \and Fuat Usta}

\institute{Department of Mathematics, University of Leicester, \\ University Road, Leicester, LE1 7RH, United Kingdom\\
{\tt \{ zd14,\ Emmanuil.Georgoulis,\ j.levesley,\ fu6 \} @le.ac.uk}  }

\maketitle

\begin{abstract}
A fast multilevel algorithm based on directionally scaled tensor-product Gaussian kernels on structured sparse grids is proposed for interpolation of high-dimensional functions and for the numerical integration of high-dimensional integrals. The algorithm is based on the recent Multilevel Sparse Kernel-based Interpolation (MLSKI) method (Georgoulis, Levesley \& Subhan, \emph{SIAM J. Sci. Comput.}, 35(2), pp.~A815--A831, 2013),  with particular focus on the fast implementation of Gaussian-based MLSKI for interpolation and integration problems of high-dimen-sional functions $f:[0,1]^d\to\mathbb{R}$, with $5\le d\le 10$. The MLSKI interpolation procedure is shown to be interpolatory and a fast implementation is proposed. More specifically, exploiting the tensor-product nature of anisotropic  Gaussian kernels, one-dimensional cardinal basis functions on a sequence of hierarchical equidistant nodes are precomputed to machine precision, rendering the interpolation problem into a fully parallelisable ensemble of linear combinations of function evaluations. A numerical integration algorithm is also proposed, based on interpolating the (high-dimensional) integrand.  A series of numerical experiments highlights the applicability of the proposed algorithm for interpolation and integration for up to 10-dimensional problems. 
\end{abstract}

\section{Introduction}
Approximation, interpolation and numerical integration of high-dimensional functions have numerous applications in applied sciences, ranging from mathematics, statistics, physics and engineering, to economics and finance. 

For low dimensional multivariate functions $f:\Omega\to\mathbb{R}$, $\Omega\subset \mathbb{R}^d$, with $2\le d\le 4$, classical interpolation and integration (cubature) methods are able to deliver accurate and efficient results to high accuracy, see e.g., \cite{davis,sobolev} and the references therein. For $d\ge 5$, standard algorithms are challenged by the, so-called, curse of dimensionality. Indeed, standard polynomial, spline or kernel-based interpolation methods of a function $f:[0,1]^d\to\mathbb{R}$, on a $d$-dimensional grid with $N$ nodes in each direction, typically require $N^d$ function evaluations to achieve rates of convergence of order $N^{-\alpha}$, for some $\alpha>0$ independent of $d$; a manifestation of the curse of dimensionality, as the interpolation problem becomes exponentially more computationally intensive with $d$.

To address this issue, a number of methods based on either randomness, especially for integration (Monte-Carlo, quasi-Monte-Carlo, multilevel Monte-Carlo; see, e.g., 
\cite{Cal98,Dick,MR2928989,MR2835612,milsteintretyakov} and the references therein) or on Smolyak/sparse grid/hyperbolic cross/boolean interpolation-type constructions (see, e.g., \cite{Smo63,DelvosFJ1982,TEM89,SchreiberAnja2000,griebel_acta_numerica} and the references therein) have appeared in the literature during the last five decades. Monte-Carlo-type approaches are, generally speaking, robust and easy to implement and their accuracy is independent of the problem dimension $d$, at the cost of typically slow convergence in a probabilistic sense \cite{Cal98,Dick}. Sparse grids and hyperbolic crosses are applicable when the function $f$ has sufficiently regular mixed derivatives; in such cases, they can deliver, typically algebraic, rates of convergence, independent of $d$ with typical complexity of $N(\log N)^{d-1}$, where $N$ is the number of points in each direction. More recently, ``$p$-version'' sparse grids, based on polynomial interpolation on Chebyshev or Clenshaw-Curtis sparse tensor-product type grids have been proposed for the calculation of high-dimensional integrals arising in the numerical approximation of elliptic boundary-value problems with random coefficients \cite{MR2318799,MR2421037,MR2646806}.

Another approach, that can be positioned within the second class of methods described above is the recent Multilevel Sparse Kernel-based Interpolation (MLSKI) method \cite{GLS}. MLSKI method is based on the concept of combination (also known as $d$-dimensional boolean interpolation \cite{DelvosFJ1982}) of smaller interpolation problems on directionally uniform nodes, whose union coincides with the sparse grid node systems produced by hierarchical linear splines. Each interpolation sub-problem is evaluated using anisotropically scaled versions of translation of standard kernels/radial basis functions. All interpolation problems are then linearly combined appropriately, leading to the Sparse Kernel-based Interpolation method (SKI). To further exploit the approximation power of the SKI interpolants, which use naturally nested sparse grid nodes, the interpolation procedure is performed in a multilevel fashion, starting from a coarse SKI interpolant $s_0$ of the function $f$, and then continuing by interpolating the residual $f-s_0$ on the next level sparse grid, and so on. We stress that, despite the fact that the MLSKI method is applied on sparse grids arising from \emph{uniform} full grids, it has not been observed to be susceptible to classical Runge-type instabilities. Hence, MLSKI can be viewed as an alternative to ``$p$-version'' sparse grids, by offering comparable convergence rates on ``uniform'' (and, therefore, hierarchical) sparse grids.

This work is concerned with further investigating the MLSKI approach both theoretically and numerically. More specifically, focusing on MLSKI using Gaussian kernels, we show that at each level of SKI, the method indeed, interpolates the function $f$. To achieve numerical exactness on all the nodes, the MLSKI approach it is sufficient to use tensor-product-type kernels, i.e., Gaussians. This also leads naturally to the fast implementation of the sub-problem computations of MLSKI \cite{GLS}, each of which can be computed completely independently. In particular, exploiting the tensor-product nature of anisotropic  Gaussian kernels, one-dimensional cardinal basis functions on a sequence of hierarchical equidistant nodes are precomputed to machine precision, rendering the interpolation problem into a fully parallelisable ensemble of linear combinations of function evaluations. A numerical integration algorithm is also proposed, based on interpolating the (high-dimensional) integrand.  A series of numerical experiments is presented, highlighting the practical applicability of the proposed algorithm for interpolation and integration for up to 10-dimensional problems.

In Section~\ref{MLSKI}, we give a brief description of the method, while, in Section~\ref{tensor}, we show that for the Gaussian kernel the cardinal functions for interpolation can be written as a product of univariate cardinal functions. In Section~\ref{tensint}, we show that if the kernel is a tensor product (such as the Gaussian) then MLSKI is also an interplatory scheme. In Section~\ref{integ}, we discuss the integration scheme with some numerical results which demonstrate the good approximation properties of the method. 

\section{Multilevel sparse kernel-based interpolation} \label{MLSKI}

We begin by briefly reviewing the MLSKI method, introduced in \cite{GLS}. To construct the sparse kernel-based interpolant, at each level, we solve a number of anisotropic radial basis function interpolation problems on appropriate direction-wise uniform sub-grids. We briefly note that due to the careful selection of anisotropic scalings for the individual interpolation problems, these are typically sufficiently well conditioned for practical computations; cf. \cite{GLS}. The individually solved interpolation problems are then linearly combined to obtain the sparse kernel-based interpolant. The multilevel sparse kernel based interpolation (MLSKI) uses a residual interpolation at different levels.

More specifically, let $ \Omega$ $ \mathrel{\mathop:}= $, $[ 0,1 ]^{d}$, $d\ge 2$, and let
$u:\Omega \rightarrow
 \mathbb{R}$. For a multi-index $\mathbf{l} = (l_{1},\ldots,l_{d})\in
\mathbb{N}^{d}$,  we define the family of directionally uniform grids
$\{\mathbb{X}_{\mathbf{l}}:\mathbf{l}\in
\mathbb{N}^{d}\}$, 
in $\Omega$, with meshsize
$h_{\mathbf{l}}=2^{-\mathbf{l}}:=(2^{-l_{1}},\ldots,2^{-l_{d}})$. Then,
 $\mathbb{X}_{\mathbf{l}}$ consists of the points
 $ \mathbf{x}_{\mathbf{l},\mathbf{i}}:= (x_{l_{1},i_{1}},\ldots,x_{l_{d},i_{d}}) $,
 with $x_{l_{j},i_{j}}=i_{j}2^{-l_{j}}$, for $ i_{j}=0,1,\dots, 2^{l_{j}}$, $j=1,\dots, d$.
 The number of nodes
$N^{\mathbf{l}}$ in $\mathbb{X}_{\mathbf{l}}$ is given by
$$
N^{\mathbf{l}}=\prod_{j=1}^{d}(2^{l_{j}}+1).
$$
If $h_{l_{j}}$ = $2^{-n},$  for all $j=1,\cdots,d,$
 $\mathbb{X}_{\mathbf{l}}$ is the uniform \emph{full grid} of level
$n$, having
 size $N$=$(2^{n}+1)^{d}$; this will be denoted by $\mathbb{X}^{n,d}$.
 
We also consider the following subset of $\mathbb{X}^{n-d+1,d}$,
\begin{equation}\label{sparse_grid_def}
\tilde{\mathbb{X}}^{n,d}:=\bigcup_{ | \mathbf{l}|_{1} = n+(d-1)}\mathbb{X}_{\mathbf{l}},
\end{equation}
with $|\mathbf{l}|_1:=l_1+\dots+l_d$, which will be referred to as the \emph{sparse grid of level $n$ in $d$ dimensions}; see Figure \ref{Fig:Sparse grid def} an illustration with $n=4$ and $d=2$.

\begin{figure}
\begin{center}
\begin{tabular}{m{1.8cm}m{0.2cm}m{1.6cm}m{0.2cm}m{1.6cm}m{0.2cm}m{1.6cm}m{0.2cm}m{1.6cm}}
   \includegraphics[width=1.9cm]{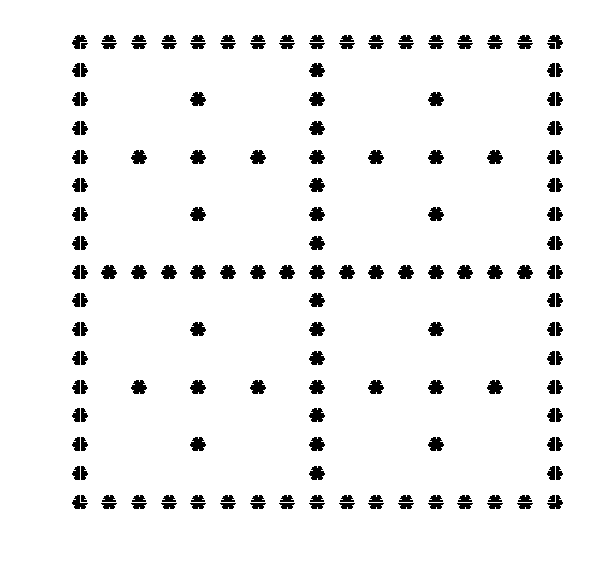} &=
      &\includegraphics[width=1.7cm]{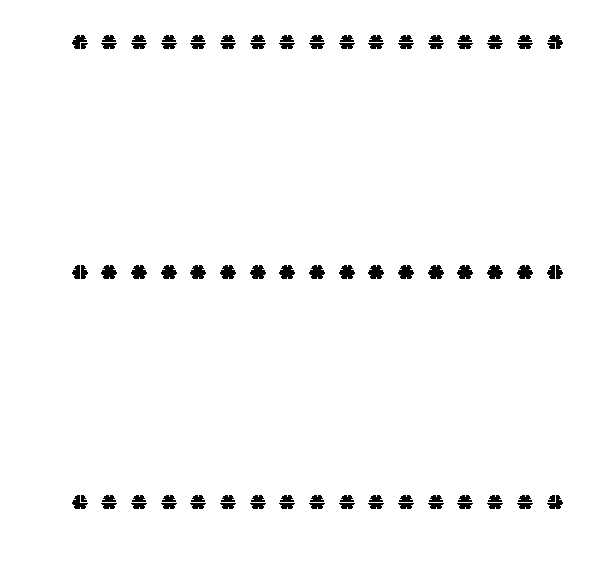} 
     & $\ \cup$  &\includegraphics[width=1.7cm]{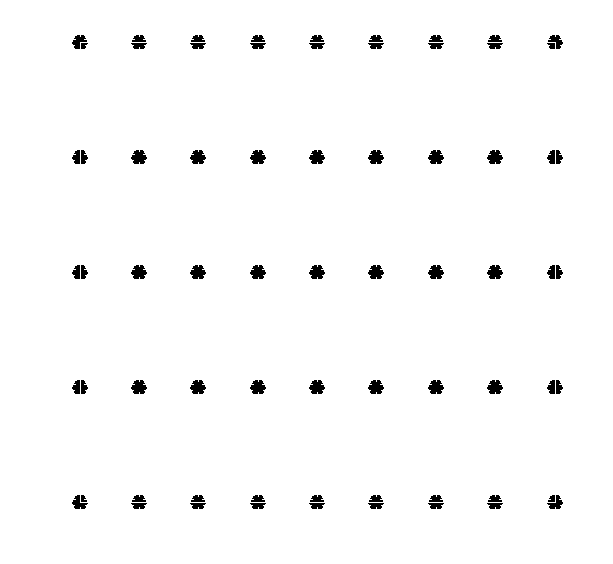} 
     & $\ \cup$ & \includegraphics[width=1.7cm]{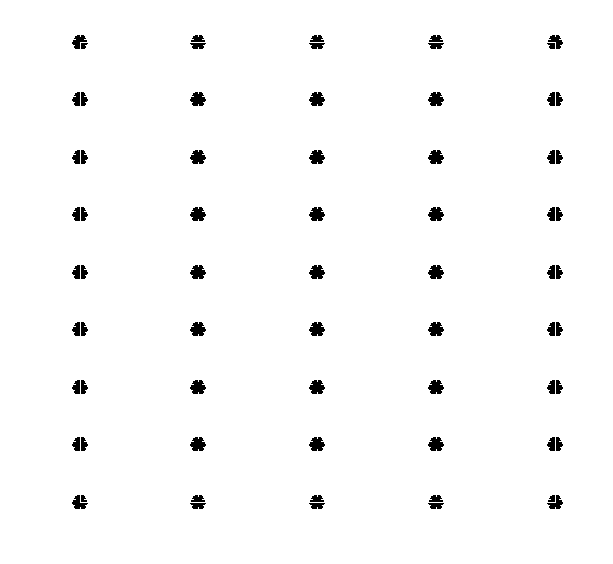} 
     & $\ \cup$ & \includegraphics[width=1.7cm]{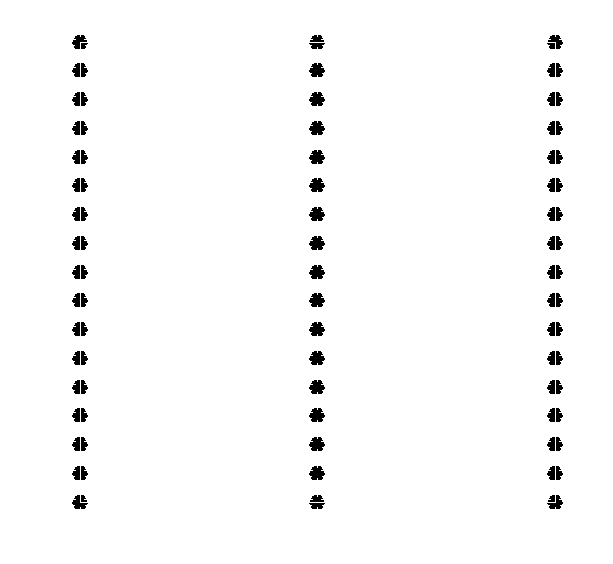}  
\end{tabular}
\end{center}
\caption{Sparse grid $\tilde{\mathbb{X}}^{4,2}$ via \eqref{sparse_grid_def}. } \label{Fig:Sparse grid def}
\end{figure}

Further, we define the transformation matrix $A_{\mathbf{l}}\in \mathbb{R}^{d \times d}$ by 
\[
A_{\mathbf{l}}:={\rm diag}(2^{l_{1}},\ldots,2^{l_{d}}).
\]
The anisotropic \emph{radial basis function} (RBF) interpolant $S_{A_{\mathbf{l}}}$ of $u$ at the points of $\mathbb{X}_{\mathbf{l}}$ is then defined by
\begin{equation}
S_{A_{\mathbf{l}}}(\mathbf{x}):=
\sum_{\mathbf{x}_{\bl,\bi}\in\mathbb{X}_{\bl}}c_{\bl,\bi}\varphi(\|
A_{\mathbf{l}}(\mathbf{x}-\mathbf{x}_{\bl,\bi})\|),
\end{equation}
for $\mathbf{x}\in\Omega$, where $c_{\bl,\bi}\in\mathbb{R}$ are chosen so that the interpolation conditions
\[
 S_{A_{\mathbf{l}}}|_{\mathbb{X}_{\mathbf{l}}}=u|_{\mathbb{X}_{\mathbf{l}}},
\]
are satisfied. Although $\phi$ can be chosen from a large family of RBF kernels (cf. \cite{GLS} for details), this work will be focused on the choice of Gaussian kernels, that is, $\phi(r)= \exp(-c^2r^2)$, for $r\ge 0$ and $c>0$, denoting the, so-called, \emph{shape-parameter}.

The \emph{sparse kernel-based interpolant} (SKI, for short) $\tilde{S}_{n,d} $ is then given by
\begin{equation}\label{Eq:SPGDinterpcombination}
\tilde{S}_{n,d}(\mathbf{x}) =\sum^{d-1}_{q=0}(-1)^{q} \left(
                                              \begin{array}{c}
                                                d-1 \\
                                                q \\
                                              \end{array}
                                            \right)
 \sum_{| \mathbf{l}|_{1}=n+(d-1)-q} S_{A_{\mathbf{l}}}(\mathbf{x}).
\end{equation}
The above combination formula has been used, for instance, for Lagrange polynomial interpolation in~\cite{DelvosFJ1982}, and  for the numerical solution of elliptic partial differential equations using the finite element method on sparse grids in ~\cite{GSZ92,GarckeandGriebel2001}. 

In Figure \ref{Fig:Sparse grid def2}, we show a visualisation of the SKI interpolant $\tilde{S}_{4,2}$ as a linear combination of the constituent sub-grid interpolants.
\begin{figure}
\begin{center}
\begin{tabular}{m{1.8cm}m{0.2cm}m{1.8cm}m{0.2cm}m{1.8cm}m{0.2cm}m{1.8cm}m{0.2cm}m{1.8cm}}
   \includegraphics[width=1.9cm]{gridS4.png} &=
      &\includegraphics[width=1.8cm]{gridX41.png} 
     & $\oplus$  &\includegraphics[width=1.8cm]{gridX32.png} 
     & $\oplus$ & \includegraphics[width=1.8cm]{gridX23.png} 
     & $\oplus$ & \includegraphics[width=1.8cm]{gridX14.png}  \\
     & $\ominus$  &\includegraphics[width=1.8cm]{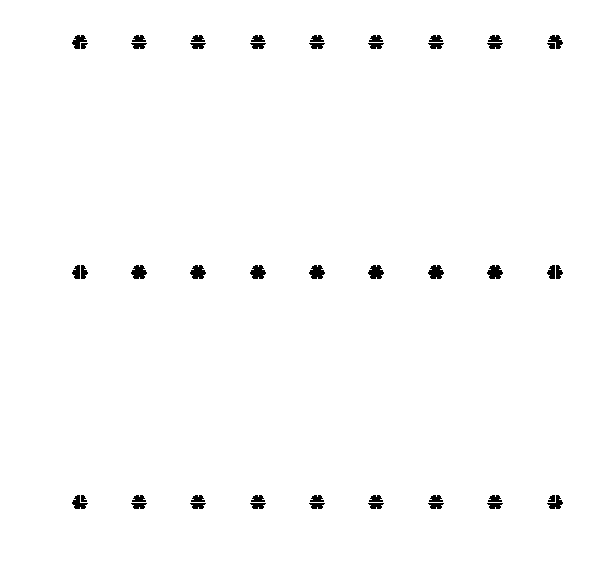} 
     & $\ominus$ & \includegraphics[width=1.8cm]{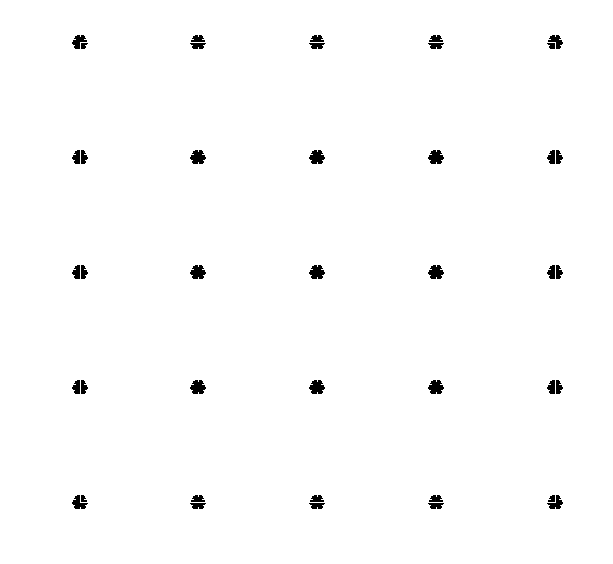} 
     & $\ominus$ & \includegraphics[width=1.8cm]{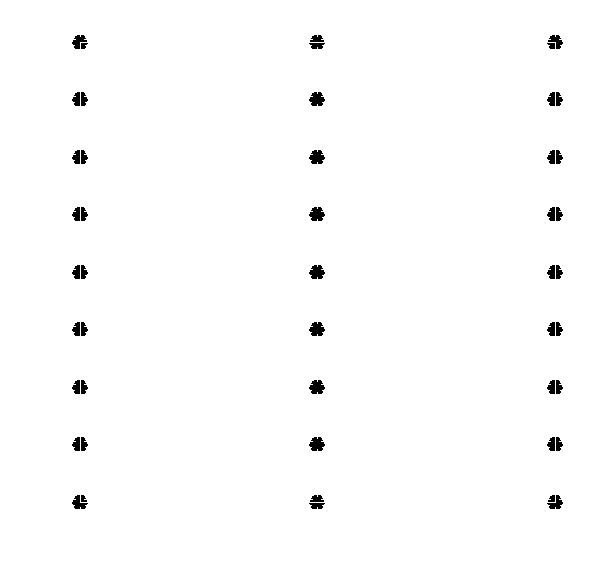}  &
\end{tabular}
\end{center}
\caption{The construction of $\tilde{S}_{4,2}$ interpolant on $\tilde{\mathbb{X}}^{4,2}$.} \label{Fig:Sparse grid def2}
\end{figure}
In the next section, we shall show that this procedure indeed results to an interpolation method, when $\phi$ is Gaussian.

The convergence of SKI was investigated in \cite{GLS}, where it was found that, although it results to acceptable approximation power when interpolating simpler tensor-product-type functions $f$, it can be also prone to very slow convergence on more challenging functions. This appears to be due to the fact that, in contrast with standard sparse grid/hyperbolic cross/Smolyak-type polynomial or linear spline interpolants, the approximation space of the SKI of level $n$ does \emph{not} contain the respective approximation spaces of levels $1,\dots, n-1$ as subspaces. In short, the SKI interpolation is not hierarchical in each level. 

To overcome this, a multilevel approach to SKI was proposed in \cite{GLS}. There is a growing volume of literature on multilevel methods for RBFs, e.g. \cite{Floater&Iske1996,Iske2001,Wendland1998,Fasshauer&Jerome1999,Fasshauer99b,Narcowich&Schaback&Ward1999,Hales&Levesley2002,IskeANDLeveslely2005} . The idea is to interpolate a high frequency residual with a rougher interpolant. 

The setting of SKI is naturally suited to be used within a multilevel interpolation algorithm. Firstly, the sparse grids from lower to higher level are nested, i.e., $\tilde{\mathbb{X}}^{n,d}\subset \tilde{\mathbb{X}}^{n+1,d}$ for $n\in\mathbb{N}$. Secondly,  each sub-grid interpolant  uses appropriately scaled anisotropic basis function with the scaling being proportional to density of the corresponding constituent sub-grid. Finally, due to the geometrical progression in the problem size from one sparse grid to the next, a multilevel algorithm would not affect adversely the attractive complexity properties of SKI.

The \emph{multilevel SKI} (MLSKI, for short) algorithm is initialised by computing the SKI $\tilde{S}_{n_0^{},d}$ on the coarsest sparse grid $\tilde{\mathbb{X}}^{n_0^{},d}$ and set $\Delta_0 :=\tilde{S}_{n_0^{},d}$. Then, for $k=1,\dots n$, $\Delta_k$ is the sparse grid interpolant to the residual $u-\sum_{j=0}^{k-1}\Delta_j$ on $\tilde{\mathbb{X}}^{k,d}$. The resulting multilevel sparse kernel based interpolant is then given by 
\[
\tilde{S}^{\rm ML}_{n,d} : = \sum_{j=0}^{n}\Delta_j.
\]

Below we describe how, if the kernel is a tensor product, we can construct a tensor product basis for the interpolation once and for all, and so avoid the costly solution of subsystems of equations, which is, in principle, highly scalable in a parallel computing architecture.

\section{Tensor Products of Univariate Cardinal Functions} \label{tensor}

Gaussian kernels on $\mathbb{R}^d$ can be viewed as tensor products of univariate Gaussians; this is evident also for the anisotropic versions of Gaussians, such as the ones discussed in the previous section. This crucial property of Gaussians will be instrumental in both the proof of interpolation of the MLSKI algorithm, and the development of fast procedures to evaluate the respective MLSKI interpolants. 

We begin by considering the set of cardinal (also known as Lagrange) functions for each interpolation sub-problem in the SKI interpolant $S_{A_{\mathbf{l}}}(\cdot)$ in \eqref{Eq:SPGDinterpcombination}, with the particular choice of Gaussian RBFs on the set $\mathbb{X}_{\mathbf{l}}$.  Let $\chi_{\bl,\bx_{\bl,\bi}}$ be the cardinal function on the grid indexed by $\bl$ for the point $\bx_{\bl,\bi} \in \mathbb{X}_{\bl}$. Let also $\chi_{l_j,x_{l_j,i_j}}$ denote the univariate cardinal function in one variable for the node $x_{l_j,i_j}$, with respect to the set of nodes $\{x_{l_j,i_j}:i_j = 0,1,\dots,2^{l_j} \}$. Hence, for $j=1,\cdots,d$, there exist $\gamma_{l_j,i_j}\in\mathbb{R}$, such that
$$
\chi_{l_j,x_{l_j,i_j}} (y) = \sum_{i_j=0}^{2^{l_j}} \gamma_{l_j,i_j} \exp(-c^2h_{l_j}^2 (y-x_{l_j,i_j})^2),
$$
for $y\in[0,1]$. Then, we have
\[
z(\by) := \prod_{j=1}^d \chi_{l_j,x_{l_j,i_j}} (y_j)  = \sum_{\bx_{\bl,\bi} \in \mathbb{X}_{\bl}} \gamma_{\bl,\bi} \exp(-c^2\sum_{j=i}^d h_{l_j}^2 (y_i-x_{l_j,i_j})^2),
\]
where $\gamma_{\bl,\bi}= \prod_{j=1}^d \gamma_{l_j,i_j} $. This is exactly the form of the cardinal function based on the points in the grid $\mathbb{X}_\bl$, which implies $ z(\by) = \chi_{\bl,\bx_{\bl,\bi}}$, due to the uniqueness of Gaussian interpolation.

Hence, it is possible to compute the cardinal functions for multivariate approximation by computing {\sl ab initio} (to arbitrarily high precision, e.g., by using symbolic calculators) the cardinal functions for univariate approximation up to (for instance) $5, 9, 17,  \cdots ,129$ equally spaced points, and store these. The approximation process would then require no solution of linear systems, massively increasing the speed of the algorithm. This algorithm will be implemented below in the numerical experiments.

\section{Tensor product kernels give interpolatory schemes}  \label{tensint}

We shall show that the combination formula~\eqref{Eq:SPGDinterpcombination} is, indeed, an interpolant. To highlight the key ideas of the proof, we first consider an example in three dimensions. A two dimensional example is too straightforward, and a four dimensional one is already too complicated.

Setting, for instance, $d=3$ and $n=7$, we seek to compute the value of the sparse kernel-based interpolant $\tilde{S}_{3,7}$ to the function $f$ at the point $a=(1/4,1/8,1/16)=(2^{-2},2^{-3},2^{-4})$, which first appears in the grid with multi-index $\bl=(2,3,4)$. From~\eqref{Eq:SPGDinterpcombination}, we see that for $d=3$, the sub-grids of the previous two levels are linearly combined to give $\tilde{S}_{3,7}$.  Hence, the set of multi-indices which index the approximation are $\{ \alpha: |\alpha|=7,8,9 \}$. 

We decompose these sets of multi-indices into two groups: 
The first sets of grids, let us call them Group 1, have two of the three components of the multi-index less than the corresponding components for the multi-index for the point, e.g., the grids represented by the multi-indices $\{ \alpha=(k,1,1) : k=5,6,7 \}$. The remaining grids, let us call them Group 2, have only one component less than the corresponding component for the point; for instance the family of grids $\{\alpha=( k,l,2):k+l \le 7, \; k \ge 2, \;  l \ge 3\}$.

We shall now study the values of the function $f$ on a typical set of grids from Group 1, $\{ (k,1,1) : k=5,6,7 \}$. To this end, we consider the cardinal functions from points on these grids, and their values at the point $a$. Let $b=(r_1 2^{-k},r_2/2,r_3/2)$, for some $r_i\in\mathbb{N}_0$, $i=1,2,3$, where $0 \le r_1 \le 2^k$, $0 \le r_2,r_3 \le 2$, with cardinal function $\chi_{(k,1,1),b}$. Then 
$$\chi_{(k,1,1),b}(a) = \chi_{k,r_1 2^{-k}}(1/4) \chi_{1,r_2/2}(1/8)\chi_{1,r_3/2}(1/16) = 0,
$$ 
if $r_1 2^{-k} \neq 1/4$, since $k \ge 2$. If $r_1 2^{-k} = 1/4$ then 
$$
\chi_{(k,1,1),b}(a) = \chi_{1,r_2/2}(1/8)\chi_{1,r_3/2}(1/16)
$$ 
which is independent of $k$. 

Thus $(1/4,r_2/2,r_3/2)$ is a node on all grids $\{ \mathbb{X}_{(k,1,1)} : k=5,6,7 \}$, and the value of the cardinal functions $\chi_{(k,1,1),b}$ are identical at $a$. Hence, the contribution to the interpolant from $f(b)$, on these grids, is
\begin{eqnarray*}
&& \sum^{2}_{q=0}(-1)^{q} {2 \choose q} f(b) \chi_{(7-q,1,1),b}(a) \\ 
& = & f(b) \chi_{1,r_2/2}(1/8)\chi_{1,r_3/2}(1/16) \sum^{2}_{q=0}(-1)^{q} {2 \choose q} \\
& = & 0.
\end{eqnarray*}
If we combine the results of the last two paragraphs we see that the contribution to the interpolant from all of the points on grids from Group 1 is 0. 

We now turn to grids from Group 2. Let us consider points on the grids $\mathbb{X}_{( k,l,2)}$ for $ k+l \le 7$, $k \ge 2$, and $  l \ge 3 $, and their values at $a$. Let $b=(r_1 2^{-k},r_2 2^{-l} ,r_3/4)$, for some $0 \le r_1 \le 2^k$, $0 \le r_2 \le 2^l$, and $0 \le r_3 \le 4$ with cardinal function $\chi_{(k,l,2),b}$. Then, 
\[
\chi_{(k,l,2),b}(a) = \chi_{k,r_1 2^{-k}}(1/4) \chi_{l,r_2 2^{-l}}(1/8)\chi_{2,r_3/4}(1/16) = 0,
\]
unless $r_1 2^{-k} = 1/4$ and $r_2 2^{-l} = 1/8$, since $k \ge 2$ and $l \ge 3$. If  $r_1 2^{-k} = 1/4$ and $r_2 2^{-l} = 1/8$, then $\chi_{(k,l,2),b}(a)=\chi_{2,r_3/4}(1/16)$, which is independent of $k$ and of $l$.

Thus, $(1/4,1/8,r_3/4)$ is contained on all grids $\mathbb{X}_{( k,l,2)}$ for $ k+l \le 7$, $k \ge 2$, and $  l \ge 3 $, and the value of the cardinal functions $\chi_{(k,l,2),b}$, for all the permissible values of $k$ are identical at $a$. Hence, the contribution to the interpolant from $f(b)$, on these grids, is
\begin{eqnarray*}
\lefteqn{\sum^{2}_{q=0}(-1)^{q} {2 \choose q} f(b) \sum_{k+l \le 7-q, \; k \ge 2, \;  l \ge 3} \chi_{(k,l,2),b}(a)} && \\
 & = & f(b) \chi_{2,r_3/4}(1/16) \sum^{2}_{q=0}(-1)^{q} {2 \choose q}  \card \{(k,l) : k+l \le 7-q, k \ge 2, \;  l \ge 3 \} \\
& = & f(b) \chi_{2,r_3/4}(1/16) \sum^{2}_{q=0}(-1)^{q} {2 \choose q}  (3-q) \; = \; 0.
\end{eqnarray*}
Therefore, the contribution to the interpolant from all grids in Group 2 is also zero. 

Thus, the only grid contributing to the interpolant is $\mathbb{X}_{(2,3,4)}$, and the only cardinal function from that grid which takes non zero values at $a$ is the cardinal function at $a$ itself. Therefore, the SKI $\tilde{S}_{7,3}$ is, indeed, an interpolant at all points of $\tilde{\mathbb{X}}^{7,3}$. Hopefully, this example highlights clearly the key role that tensor-product nature of the kernel has in the proof of interpolation property. 

Equipped with the insight gained by the above example, we shall now consider the general case. To this end, we begin with the following counting result.
\begin{lemma} \label{poly}
Let $\bk \in \mathbb{N}^d$ and $|\bk| < p\in\mathbb{N}$. Then, 
\[
\card \{ \bj \in \mathbb{N}^d : \bj \ge \bk, \; |\bj| = p \} = {p-|\bk|+d-1 \choose d-1}.
\]
\end{lemma}

\begin{proof} For $d=1$, the result is immediate, since for $k< p$,
\[
 \card \{ j \in \mathbb{N} : j \ge k, \; j = p \}=1.
 \]
  For $d\ge 2$, let us write $\bk = (k_1,\tilde \bk)$, and consider we set , and consider $\{ \tilde \bj \in \mathbb{N}^{d-1} : |\tilde \bj | = p-j_1 \}$, for $j_1=k_1,k_1+1,\cdots,p-|\tilde k|$. Then, by induction,
\begin{eqnarray*}
\card \{ \tilde{\bj}\in\mathbb{N}^{d-1}  : \tilde \bj \ge \tilde \bk, \; | \tilde \bj | = p-j_1 \} & =&  {p-j_1-|\tilde \bk|+d-1 \choose d-1}.
\end{eqnarray*}
Thus, 
\begin{eqnarray*}
\card \{ \bj \in \mathbb{N}^d : \bj \ge \bk, \; |\bj| = p \}  & =& \sum_{j_1=k_1}^{p-|\tilde \bk|}  \card \{ \tilde \bj  : \tilde \bj \ge \tilde \bk, \; | \tilde \bj | = p-j_1 \}   \\
& = & \sum_{j_1=k_1}^{p-|\tilde \bk|}  {p-j_1-|\tilde  \bk|+d-1 \choose d-1} \\
& = &  \sum_{j_1=1}^{p-|\bk|+1}  {j_1+d-2 \choose d-1}  \\
& = & {p-|\bk|+d \choose d},
\end{eqnarray*}
using the well-known summation formula for binomial coefficients; e.g., \cite[1.2.2]{ab}. $\Box$
\end{proof}

We are now ready to state and prove the main result of this section.

\begin{theorem}
Assuming that the interpolation kernel has the form
\[
\psi(\by) = \prod_{i=1}^d \phi(y_i),
\] then SKI with this kernel is an interpolatory scheme.
\end{theorem}

\begin{proof}
We wish to compute the value of the sparse grid interpolant $\tilde{S}_{d,n}$ to the function $f$ at the point $\ba=(a_12^{-l_1},a_22^{-l_2},\cdots,a_d2^{-l_d})$, with $0 \le a_j \le 2^{l_j}$, $1 \le j \le d$, with at least one of the $a_j\in\mathbb{N}_0$ odd (this ensures that $\mathbb{X}_\bl$ is the grid with the smallest index in which this points appears). From hypothesis, the kernel is of tensor-product form, which implies the cardinal functions for interpolation are of the form
$$
\chi_{\bm,\bx} = \prod_{j=1}^d \chi_{m_j,x_j 2^{-m_{j}}},
$$
as observed in Section~\ref{tensor}.

In order to create a decomposition of the indices as in the example above, we need to introduce some notation. Let $d_1,d_2 \in \mathbb{N}$, with $d_1+d_2=d$. Let $\bn_1 \in \mathbb{N}^{d_1}$  with $\bn_1(j) \in \{1,2,\cdots,d\}$, $j=1,2,\cdots,d_1$, and $\bn_1(j) < \bn_1(j+1)$, $j=1,2,\cdots,d_1-1$. Similarly,  let $\bn_2 \in \mathbb{N}^{d_2}$  with $\bn_2(j) \in \{1,2,\cdots,d\}$, $j=1,2,\cdots,d_2$, and $\bn_2(j) <  \bn_2(j+1)$, $j=1,2,\cdots,d_2-1$.  Additionally, $\bn_1(j) \neq \bn_2(k)$ for any $j,k$. In other words, the components of $\bn_1$ and $\bn_2$ exhaust the set $\{1,2,\cdots,d\}$, and these numbers are ordered within the vectors $\bn_1$ and $\bn_2$.  We should note that once $\bn_1$ is specified, $\bn_2$ is uniquely determined and vice versa.

Let $\bm_i \in \mathbb{N}^{d_i}$, $i=1,2$. Then, let the multi-index $\bm = (\bn_1,\bm_1,\bn_2,\bm_2) \in \mathbb{N}^d$ have components $\bm(k) = \bm_i(j)$ if $\bn_i(j)=k$, $i=1,2$. So, for instance, if $\bn_1(3)=7$, then $\bm(7)=\bm_1(3)$, and if $\bn_2(4)=5$, then $\bm(5)=\bm_2(4)$. In this way, we break multi-indices into two pieces in a convenient fashion.

On the other hand, for $\bm \in \mathbb{N}^d$, let $\bm_{\bn_1} \in \mathbb{N}^{d_1}$ with $\bm_{\bn_1}(i)=\bm(\bn_1(i))$, $i=1,2,\cdots,d_1$ and $\bm_{\bn_2} \in \mathbb{N}^{d_2}$ with $\bm_{\bn_2}(i)=\bm(\bn_2(i))$, $i=1,2,\cdots,d_2$. Then, we have the identity
$$
\bm = (\bn_1,\bm_{\bn_1},\bn_2,\bm_{\bn_2}).
$$

Now, for each $d_1 = 1,2,\cdots,d-1$, and $\bn_1 \in \mathbb{N}^{d_1}$, let
$$
I(\bl,d_1,\bn_1)=\{ \bm_1 : \bm_1 < \l_{\bn_1} \}.
$$
For each $\bm_1 \in I(\bl,d_1,\bn_1)$, define
$$
J(\bl,d_1,\bn_1,\bm_1)=\{ (\bn_1,\bm_1,\bn_2,\bm_2) : \bm_2 \in \mathbb{N}^{d_2}, \; \bm_2 \ge \bl_{\bn_2}  \}.
$$
The sets $J(\bl,d_1,\bn_1,\bm_1)$ partition all multi-indices into sets with a fixed set of components of the multi-index less than those of $\bl$ and the remaining components greater than or equal to those of $\bl$. The only multi-index missing from this set is $\bl$ itself.

We compute the cardinality of the subsets of $J(\bl,d_1,\bn_1,\bm_1)$, given by
$$
J(\bl,d_1,\bn_1,\bm_1,q)=\{ \bm \in J(\bl,d_1,\bn_1,\bm_1) : |\bm|=n+d-1-q \}, 
$$
for $q=0,1,\cdots,d-1$, which, using Lemma~\ref{poly}, is given by
\begin{eqnarray}
\card J(\bl,d_1,\bn_1,\bm_1,q)  & = & {n+d+d_1-2-|\bl_{\bn_2}| -q \choose d_1-1}. \label{cardsub}
\end{eqnarray}

We now consider the contribution to the interpolant from a typical point on one of the grids indexed by elements of $J(\bl,d_1,\bn_1,\bm_1)$. 

Let $\bm=(\bn_1,\bm_1,\bn_2,\bm_2) \in J(\bl,d_1,\bn_1,\bm_1)$. Then, by definition, $\bl_{\bn_2} \ge \bm_2$. Let $\bx=(x_1 2^{-m_1},x_2 2^{-m_2},\cdots x_d 2^{-m_d})$, for some $0 \le x_j \le 2^{m_j}$, $i=1,2,\cdots,d$,  with cardinal function $\chi_{\bm,\bx}$. Then, 
$$
\chi_{\bm,\bx}(\ba) = \prod_{i=1}^d \chi_{m_i,x_i2^{-m_i}}(a_i 2^{-l_i}).
$$
Hence, $\chi_{\bm,\bx}(\ba) \neq 0$ only if $x_i 2^{-m_i} = a_i2^{-l_i}$, $i=\bn_2(j)$, $j=1,2,\cdots,d_2$. In this case,
$$
\chi_{\bm,\bx}(\ba) = \prod_{j=1}^{d_1} \chi_{{m_{\bn_1(j)}},x_{\bn_1(j)} 2^{-m_{\bn_1(j)}}}(a_{\bn_1(j)} 2^{-l_{\bn_1(j)}}),
$$
which is independent of $\bm_2$. 

So we have, for fixed $\bn_1$ and $\bm_1$, the same contribution at $\ba$ from any individual point which appears in a grid in $J(\bl,d_1,\bn_1,\bm_1)$. Thus, the contribution to from $\bx$ at point $\ba$ is
\begin{eqnarray*}
\lefteqn{\sum^{d-1}_{q=0}(-1)^{q} {d-1 \choose q} f(\bx) \sum_{\bm \in J(\bl,d_1,\bn_1,\bm_1,q)} \chi_{\bm,\bx}(\ba)} && \\
 & = & f(\bx)  \prod_{j=1}^{d_1} \chi_{{m_{\bn_1(j)}},x_{\bn_1(j)} 2^{-m_{\bn_1(j)}}}(a_{\bn_1(j)} 2^{-l_{\bn_1(j)}})  \\
 && \hspace{2cm} \times \sum^{d-1}_{q=0}(-1)^{q} {d-1 \choose q} \card J(\bl,d_1,\bn_1,\bm_1,q) \\
& = &  f(\bx)  \prod_{j=1}^{d_1} \chi_{{m_{\bn_1(j)}},x_{\bn_1(j)} 2^{-m_{\bn_1(j)}}}(a_{\bn_1(j)} 2^{-l_{\bn_1(j)}}) \\
&& \hspace{2cm} \times \sum^{d-1}_{q=0}(-1)^{q} {d-1 \choose q} {n+d+d_1-2-|\bl_{\bn_2}| -q \choose d_1-1} \\
& = & 0,
\end{eqnarray*}
since the binomial coefficient is a polynomial in $q$ of degree less than $d-1$, which is thus annihilated by the difference operator.

Thus, we see that all contributions from points in other grids other than $X_\bl$ are zero at points in $X_\bl$. Clearly, the only contribution at $\ba \in X_\bl$ from points in $X_\bl$ will be from $\ba$ itself. $\Box$
\end{proof}

\section{Numerical Integration} \label{integ}

Equipped with an interpolation method, we can now also discuss integration over the unit cube. A quadrature rule will approximate the integral
$$
\int_{[0,1]^d} f(\bx) d\bx \approx \sum_{\bz \in Z} w_\bz f(\bz),
$$
where $Z \subset [0,1]^d$ is a finite set of points, and $\{ w_\bz : \bz \in Z \}$ are the {\sl weights}. 
If we have a cardinal basis for $Z$, $\{\chi_\bz : \bz \in Z \}$, then we approximate
$$
f = \sum_{\bz \in Z} f(\bz) \chi_\bz,
$$
giving rise to the quadrature formula
$$
\int_{[0,1]^d} f(\bx) d\bx \approx \sum_{\bz \in Z} f(\bz) \int_{[0,1]^d} \chi_\bz(\bx) d\bx.
$$
Thus, the weights are given by
$$
\int_{[0,1]^d} \chi_\bz(\bx) d\bx.
$$

In the case when the cardinal functions are tensor products, viz., $\chi_\bz=\chi^1_{z_1} \times \cdots \times \chi^d_{z_d}$ (as is the case of Gaussians, as observed in Section~\ref{tensor}), we can straightforwardly compute the weights as follows:
$$
\int_{[0,1]^d} \chi_\bz(\bx) d\bx = \prod_{j=1}^d \int_0^1 \chi^j_{z_j}(x_j) dx_i.
$$

As also mentioned in Section~\ref{tensor}, we can compute the cardinal functions offline to arbitrary precision for, e.g., up to 129 equally spaced points, thereby allowing for fast and highly accurate computation of the quadrature weights. Indeed, the integral of a univariate cardinal function is of the form
\[
 I = \int_{0}^{1}\chi(x)dx  
   = \int_{0}^{1}\sum_{i=1}^N c_{i} \exp(-r^2(x-y_i)^2)dx,
\]
for some $r, y_i,c_i \in \mathbb{R}$, $i=1,\cdots,N$, which, upon the change of variables $ u=r(x-y_i) $, gives
\[
 I   = {1 \over r} \sum_{i=1}^N c_{i} \int_{-ry_i}^{1-ry_i} \exp(-u^2)du 
  =  {\pi \over 2r} \sum_{i=1}^N c_{i} (\erf (r(1-y_i)-\erf(-ry_i)),
 \]
with the error function $\erf$ defined by
\begin{equation*}
\erf(x)=\frac{2}{\sqrt{\pi}}\int_{0}^{x}e^{-t^{2}}dt.
\end{equation*}
The error function can be computed to arbitrary precision.


\subsection{Examples}

In this section we give a number of examples. We would like to develop a set of benchmark examples so we invite other researchers to contact us with their results for these cases, and perhaps forward us their own examples. We do two tensor product examples, respectively in 5 and 10 dimensions, and then two non-tensor product examples, the first being smooth, the second with derivative singularities on the boundary. The motivation for the second choice is approximation of the pay-off function in option pricing problems.

\begin{enumerate}
\item Consider the function
\begin{equation}
f(\bx)=\prod_{i=1}^5 4x_i(1-x_i).
\end{equation}
The corresponding exact integral of this function on the domain $[0,1]^5$ with $16$ digits accuracy is $(\frac{2}{3})^5=0.131687242798354$. In Table~\ref{5dex} we give the MLSKI quadrature computation of the same integral, recording the number of nodes use, and the absolute and relative errors.

\begin{table}[!htb] 
\begin{center}
\addtolength{\tabcolsep}{-1pt}
\begin{tabular}{||c|c|c|c|c||}
  \hline
nodes & Absolute error &  Relative error \\
\hline \hline
  243          & 3.0091e-2 & 2.2850e-1 \\
  1053        & 5.1232e-3 & 3.8904e-2  \\
  3753      & 1.3013e-3 & 9.8818e-3  \\
  12033      & 1.4927e-4 & 1.1335e-3  \\
  36033     & 3.6134e-5 & 2.7439e-4  \\
  102785     & 3.4530e-6 & 2.6222e-5  \\
  282525    & 8.1811e-7 & 6.2125e-6  \\
  754845   & 6.9041e-8 & 5.2428e-7  \\[0.5ex]
 \hline
\end{tabular}
\end{center}
\caption{MLSKI quadrature for $f$.}
\label{5dex}
\end{table}
\item The following example is in 10 dimensions:
\begin{equation}
g(\bx)=\prod_{i=1}^{10}e^{(-x_i(1-x_i))}
\end{equation}
The corresponding exact integral of this function on the domain $[0,1]^{10}$ with $16$ digits accuracy is $0.194279067580947$. Table~\ref{10dex} shows how the error reduces with respect to the increase of degrees of freedom. 
\begin{table}[!htb]
\begin{center}
\addtolength{\tabcolsep}{-1pt}
\begin{tabular}{||c|c|c|c|c||}
  \hline
nodes  & Absolute-error &  Relative-error \\
\hline \hline
 
  59049    & 1.5068e-1 & 7.7556e-1  \\
  452709    & 5.8153e-3& 2.9933e-2 \\
  2421009   &3.5882e-3 & 1.8469e-2  \\
 10819089  & 4.9348e-4 &  2.5400e-3  \\[0.5ex]
 \hline
\end{tabular}
\end{center}
\caption{MLSKI quadrature for $g$.} 
\label{10dex}
\end{table}
\item Let us now consider the non-tensor product Franke $4$D function:
\begin{equation}
\begin{aligned}
   u_{F4D}(\bx)  = &\; \frac{3}{4}e^{(-(9x_{1}-2)^{2}-(9x_{2}-2)^{2} -(9x_{3}-2)^{2})/4 -(9x_{4}-2)^{2})/8}\\
     &  +\frac{3}{4}e^{(-(9x_{1}+1)^{2})/49-((9x_{2}+1)^{2})/10-((9x_{3}+1)^{2})/29-((9x_{4}+1)^{2})/39}\\
      & +\frac{1}{2}e^{(-(9x_{1}-7)^{2})/4-(9x_{2}-3)^{2}-((9x_{3}-5)^{2})/2-((9x_{4}-5)^{2})/4}\\
       & -\frac{1}{5}e^{(-(9x_{1}-4)^{2})/4-(9x_{2}-7)^{2}-((9x_{3}-5)^{2})-((9x_{4}-5)^{2})}.
\end{aligned} 
 \end{equation}
The integral of this function on the domain $[0,1]^4$ with $16$ digits accuracy is $0.037221856819405$.  
The errors in integration are shown in Table~\ref{4dex}.
\begin{table}[!htb]
\begin{center}
\addtolength{\tabcolsep}{-1pt}
\begin{tabular}{||c|c|c|c|c||}
  \hline
nodes  & Absolute-error &  Relative-error \\
\hline \hline
  81          & 1.6398e-2 & 4.4055e-1\\
  297        & 1.2736e-2& 3.4216e-1 \\
  945       & 7.9106e-3 & 2.1253e-1 \\
  2769      & 5.4904e-3 & 1.4751e-1\\
  7681      & 5.5825e-4 & 1.4998e-2  \\
  20481     & 1.3012e-4 & 3.4959e-3  \\
  52993    & 1.6245e-5 & 4.3643e-4  \\
  133889   & 1.2027e-7 & 3.2312e-6  \\
  331777  & 2.2934e-8 &  6.1615e-7  \\[0.5ex]
 \hline
\end{tabular}
\end{center}
\caption{MLSKI quadrature for $U_{F4D}$.}
\label{4dex}
\end{table}
\item The final example is a $5$ dimensional non-tensor product function with derivative discontinuities in different directions on the boundary. In order to apply MLSKI to derivative pricing problems we would need to be able to approximate such functions as
\begin{equation}
F_{{\rm payoff}}(\bx)=\sum_{i=1}^5 \max  (x_i-\frac{1}{2},0).
\end{equation}
The exact integral of this function on the domain $[0,1]^5$  is $\frac{5}{8}$. In Table~\ref{payoffex} we see how the error behaves with regard to the degrees of freedom. 
\begin{table}[!htb]
\begin{center}
\addtolength{\tabcolsep}{-1pt}
\begin{tabular}{||c|c|c|c|c||}
  \hline
nodes & Absolute error &  Relative error \\
\hline \hline
  243          & 1.5129e-1 & 2.4206e-1 \\
  1053        & 5.4282e-3 & 8.6851e-3  \\
  3753      & 2.9705e-3  & 4.7529e-3\\
  12033      & 1.0128e-3 & 1.6206e-3  \\
  36033     & 3.2119e-4 & 5.1390e-4 \\
  102785     & 9.0693e-5 & 1.4511e-4  \\
  282525    & 2.2032e-5& 3.5251e-5  \\
  754845   & 5.7779e-6& 9.2447e-6 \\[0.5ex]
 \hline
\end{tabular}
\end{center}
\caption{MLSKI quadrature for $F_{{\rm payoff}}$.} 
\label{payoffex}
\end{table}
\end{enumerate}

\bibliographystyle{siam}
\bibliography{Thesisbibdatanew}

\def\cprime{$'$}
\begin{thebibliography}{10}

\bibitem{ab}
{\sc A.~Abramowitz and I.~Stegun}, {\em Handbook of Mathematical Functions},
  National Bureau of Standards, 1964.

\bibitem{MR2318799}
{\sc I.~Babu{\v{s}}ka, F.~Nobile, and R.~Tempone}, {\em A stochastic
  collocation method for elliptic partial differential equations with random
  input data}, SIAM J. Numer. Anal., 45 (2007), pp.~1005--1034.

\bibitem{MR2646806}
\leavevmode\vrule height 2pt depth -1.6pt width 23pt, {\em A stochastic
  collocation method for elliptic partial differential equations with random
  input data}, SIAM Rev., 52 (2010), pp.~317--355.

\bibitem{griebel_acta_numerica}
{\sc H.-J. Bungartz and M.~Griebel}, {\em Sparse grids}, Acta Numer., 13
  (2004), pp.~147--269.

\bibitem{Cal98}
{\sc R.~E. Caflisch}, {\em Monte {C}arlo and quasi-{M}onte {C}arlo methods}, in
  Acta numerica, 1998, vol.~7 of Acta Numer., Cambridge Univ. Press, Cambridge,
  1998, pp.~1--49.

\bibitem{MR2835612}
{\sc K.~A. Cliffe, M.~B. Giles, R.~Scheichl, and A.~L. Teckentrup}, {\em
  Multilevel {M}onte {C}arlo methods and applications to elliptic {PDE}s with
  random coefficients}, Comput. Vis. Sci., 14 (2011), pp.~3--15.

\bibitem{davis}
{\sc P.~J. Davis}, {\em Interpolation and approximation}, Dover Publications
  Inc., New York, 1975.
\newblock Republication, with minor corrections, of the 1963 original, with a
  new preface and bibliography.

\bibitem{DelvosFJ1982}
{\sc F.-J. Delvos}, {\em {$d$}-variate {B}oolean interpolation}, J. Approx.
  Theory, 34 (1982), pp.~99--114.

\bibitem{Dick}
{\sc J.~Dick, F.~Y. Kuo, and I.~H. Sloan}, {\em High-dimensional integration:
  the quasi-{M}onte {C}arlo way}, in Acta numerica, 2013, vol.~22 of Acta
  Numer., Cambridge Univ. Press, Cambridge, 2013, pp.~133--288.

\bibitem{Fasshauer99b}
{\sc G.~E. Fasshauer}, {\em Solving differential equations with radial basis
  functions: multilevel methods and smoothing}, Adv. Comput. Math., 11 (1999),
  pp.~139--159.
\newblock Radial basis functions and their applications.

\bibitem{Fasshauer&Jerome1999}
{\sc G.~E. Fasshauer and J.~W. Jerome}, {\em Multistep approximation
  algorithms: improved convergence rates through postconditioning with
  smoothing kernels}, Adv. Comput. Math., 10 (1999), pp.~1--27.

\bibitem{Floater&Iske1996}
{\sc M.~S. Floater and A.~Iske}, {\em Multistep scattered data interpolation
  using compactly supported radial basis functions}, J. Comput. Appl. Math., 73
  (1996), pp.~65--78.

\bibitem{GarckeandGriebel2001}
{\sc J.~Garcke and M.~Griebel}, {\em On the parallelization of the sparse grid
  approach for data mining}, in Large-Scale Scientific Computations, Third
  International Conference, LSSC 2001, Sozopol, Bulgaria, S.~Margenov,
  J.~Wasniewski, and P.~Yalamov, eds., vol.~2179 of Lecture Notes in Computer
  Science, Springer, 2001, pp.~22--32.
\newblock also as SFB 256 Preprint 721, Universit\"at Bonn, 2001.

\bibitem{GLS}
{\sc E.~Georgoulis, J.~Levesley, and F.~Subhan}, SIAM J. Sci. Comput., 35
  (2013), pp.~815--831.

\bibitem{GSZ92}
{\sc M.~Griebel, M.~Schneider, and C.~Zenger}, {\em A combination technique for
  the solution of sparse grid problems}, in Iterative methods in linear algebra
  ({B}russels, 1991), North-Holland, Amsterdam, 1992, pp.~263--281.

\bibitem{Hales&Levesley2002}
{\sc S.~J. Hales and J.~Levesley}, {\em Error estimates for multilevel
  approximation using polyharmonic splines}, Numer. Algorithms, 30 (2002),
  pp.~1--10.

\bibitem{Iske2001}
{\sc A.~Iske}, {\em Hierarchical scattered data filtering for multilevel
  interpolation schemes}, in Mathematical methods for curves and surfaces
  ({O}slo, 2000), Innov. Appl. Math., Vanderbilt Univ. Press, Nashville, TN,
  2001, pp.~211--221.

\bibitem{IskeANDLeveslely2005}
{\sc A.~Iske and J.~Levesley}, {\em Multilevel scattered data approximation by
  adaptive domain decomposition}, Numer. Algorithms, 39 (2005), pp.~187--198.

\bibitem{MR2928989}
{\sc F.~Y. Kuo, C.~Schwab, and I.~H. Sloan}, {\em Quasi-{M}onte {C}arlo methods
  for high-dimensional integration: the standard (weighted {H}ilbert space)
  setting and beyond}, ANZIAM J., 53 (2011), pp.~1--37.

\bibitem{milsteintretyakov}
{\sc G.~N. Milstein and M.~V. Tretyakov}, {\em Stochastic numerics for
  mathematical physics}, Scientific Computation, Springer-Verlag, Berlin, 2004.

\bibitem{Narcowich&Schaback&Ward1999}
{\sc F.~J. Narcowich, R.~Schaback, and J.~D. Ward}, {\em Multilevel
  interpolation and approximation}, Appl. Comput. Harmon. Anal., 7 (1999),
  pp.~243--261.

\bibitem{MR2421037}
{\sc F.~Nobile, R.~Tempone, and C.~G. Webster}, {\em A sparse grid stochastic
  collocation method for partial differential equations with random input
  data}, SIAM J. Numer. Anal., 46 (2008), pp.~2309--2345.

\bibitem{SchreiberAnja2000}
{\sc A.~Schreiber}, {\em The method of Smolyak in multivariate interpolation},
  PhD thesis, der Mathematisch-Naturwissenschaftlichen Fakult\"{a}ten, der
  Georg-August-Universit\"{a}t zu G\"{o}ttingen, 2000.

\bibitem{Smo63}
{\sc S.~A. Smolyak}, {\em Quadrature and interpolation of formulas for tensor
  product of certain classes of functions}, Soviet Math. Dokl., 4 (1963),
  pp.~240--243.

\bibitem{sobolev}
{\sc S.~L. Sobolev and V.~L. Vaskevich}, {\em The theory of cubature formulas},
  vol.~415 of Mathematics and its Applications, Kluwer Academic Publishers
  Group, Dordrecht, 1997.
\newblock Translated from the 1996 Russian original and with a foreword by S.
  S. Kutateladze, Revised by Vaskevich.

\bibitem{TEM89}
{\sc V.~N. Temlyakov}, {\em Approximation of functions with bounded mixed
  derivative}, in Proc. Steklov Institute Math, 1989., AMS, 1989.

\bibitem{Wendland1998}
{\sc H.~Wendland}, {\em Numerical solution of variational problems by radial
  basis functions}, in Approximation theory {IX}, {V}ol. 2 ({N}ashville, {TN},
  1998), Innov. Appl. Math., Vanderbilt Univ. Press, Nashville, TN, 1998,
  pp.~361--368.

\end{thebibliography}

\end{document}